\documentclass[a4paper]{article}

\usepackage[utf8]{inputenc}
\usepackage[T1]{fontenc}
\usepackage[francais,english]{babel}
\usepackage{refcount}

\usepackage{enumitem}
\usepackage{amsmath,amssymb,amsfonts,amsthm}
\usepackage{mathrsfs}
\usepackage{amsfonts}
\usepackage{amssymb}
\usepackage{latexsym}
\usepackage{comment}
\usepackage[normalem]{ulem}

\usepackage{xcolor}
\definecolor{hypcolor}{RGB}{0,110,90}   
\usepackage[colorlinks=true, linkcolor=red, citecolor=red, urlcolor=blue]{hyperref}
\newcommand{\hypref}[2]{%
    \hyperlink{#1}{\textcolor{hypcolor}{\textbf{#2}}}%
}

\def\C{\mathbb C}

\def\et0{e^{tA}x_0}

\numberwithin{equation}{subsection}

\makeatletter
\renewcommand{\theequation}{%
  \ifnum\value{subsection}=0
    \mbox{\thesection.\arabic{equation}}%
  \else
    \mbox{\thesubsection.\arabic{equation}}%
  \fi
}
\makeatother

\newtheorem{thm}{Theorem}
\newtheorem{Prop}{Proposition}
\newtheorem{Def}{Definition}
\newtheorem{Lem}{Lemma}
\newtheorem{Cor}{Corollary}
\newtheorem{rk}{Remark}
\newtheorem{CE}{Counterexample}

\author{
\textsc{Shri Lal Raghudev Ram Singh}\thanks{
University of Waterloo, email: \texttt{slrrsingh@uwaterloo.ca}
},
\quad\textsc{Roberto Guglielmi}\thanks{University of Waterloo,
email: \texttt{roberto.guglielmi@uwaterloo.ca}}
\thanks{The authors acknowledge the support of the Natural Sciences and Engineering Research
Council of Canada (NSERC), funding reference number RGPIN-2021-02632.
}
}
\title{Quasi-compactness and uniform stabilization\\
on general Banach spaces under $\theta$-subordinate\\ perturbations of semigroup generators}
\date{}

\begin{document}

\maketitle

\begin{abstract}
We prove that the quasi-compactness of an analytic semigroup is preserved under $\theta$-subordinate perturbations of its generator on general Banach spaces, provided the perturbations are compact along trajectories. This allows us to prove the permanence of uniform exponential stability for an analytic semigroup under similar perturbations of its generator, provided the perturbed generator generates a strongly stable semigroup. We then show how the analyticity assumption can be relaxed to the class of Crandall-Pazy semigroups under a smaller class of $\theta$-subordinate perturbations, and to immediately differentiable semigroups under bounded perturbations. As an application of the stabilization result, we consider the one-dimensional Neumann Laplacian on a non-reflexive Banach space. Finally, we present counterexamples that demonstrate the lack of uniform stabilization under the larger class of $A$-bounded perturbations.

\end{abstract}

\bigskip
\noindent
\textbf{Key words: } 
Quasi-compact semigroups, analytic semigroups,
uniform exponential stability,
stabilization, relatively bounded perturbations.
\smallskip

\noindent
\textbf{AMS subject classifications: } Primary: 46B50, 47A55, 47D06, 47B07; Secondary: 93D23, 35B35

\section{Introduction}\label{intro}

Let $A:D(A)\subset X\to X$ be the infinitesimal generator of a strongly continuous semigroup $\left(e^{tA}\right)_{t\geq 0}$ of bounded linear operators on a Banach space $X$, equipped with the norm $\|\cdot\|$. We denote by $\|\cdot\|_{\mathcal{L}}$ the operator norm on $\mathcal{L}(X)$.

In a typical abstract stabilization problem, one starts from the open-loop system
$$
\dot{x}(t)=Ax(t)+u(t),\qquad x(0)=x_0\in X,
$$
and seeks to choose $u(t)=Bx(t)$ in such a way that the resulting closed-loop dynamics
\begin{equation}\label{closedSys}
    \dot{x}(t)=(A+B)x(t),\qquad x(0)=x_0\in X,
\end{equation}
possess a desired stability property. For an operator $B:D(B)\subset X\to X$,  let  $A+B$ generates a strongly continuous semigroup $\left(e^{t(A+B)}\right)_{t\geq 0}$ on $X$. In this paper, we discuss general assumptions on the feedback gain operator $B$ and on the semigroup $\left(e^{tA}\right)_{t\geq 0}$ under which the strong stability of the perturbed semigroup $\left(e^{t(A+B)}\right)_{t\geq 0}$ implies its uniform exponential stability, that is, the existence of constants $M\geq 1$ and $\omega>0$ such that
$$
\|e^{t(A+B)}\|_{\mathcal{L}}
\leq M e^{-\omega t},
\qquad t\geq 0.
$$

{The aforesaid stabilization problem can also be viewed as the study of the stability properties of the semigroup $\left(e^{tA}\right)_{t\geq 0}$ under perturbations. In particular, we are interested in investigating conditions on the perturbation operator $B$ and the semigroup generated by $A$ under which the uniform exponential stability of $\left(e^{tA}\right)_{t\geq 0}$ is preserved.}

{\subsection{Preliminaries} Let us recall the definitions of the following classes of unbounded perturbations.
\begin{Def}[$A$-boundedness, {\cite[Definition III.2.1]{EnNa}}]\label{Abound}
A linear operator $B:D(B)\to X$ is said to be \emph{$A$-bounded} if \( D(A) \subseteq D(B) \), and there exist constants $a, b \geq 0$  such that 
\begin{equation}\label{Aboundedness}
\|Bx\| \;\le\; a\,\|Ax\| \;+\; b\,\|x\|,\qquad \forall x\in D(A).
\end{equation}
The $A$-bound of $B$ is defined by
$$ a_0:=\inf\left\{
a\geq 0:\ \text{there exists } b\geq0
\text{ such that \eqref{Aboundedness} holds}
\right\}.
$$ In standard perturbation theory, some authors refer to $A$-bounded perturbations as \emph{relatively bounded perturbations} (see, for instance, \cite{Kato,Pazy}).
\end{Def}
\begin{Def}[$\theta$-subordination {\cite[Chapter 1]{Mar88}}]\label{theta-sub}
A linear operator $ B : D(B) \to X $ is said to be \emph{$\theta$-subordinate} to $A$ if $D(A) \subseteq D(B)$  and, for some \( 0 \leq \theta < 1 \), there exists a constant $c>0$ such that
\[
\| Bx \| \le c\, \| Ax \|^{\theta} \| x \|^{1 - \theta}, \qquad \forall x \in D(A).
\]
\(B\) is also termed \emph{subordinate of order} \(\theta\) in the literature (see, for example, \cite[Definition~7.2]{Krein1971}).
\end{Def}
\begin{Def}[$(-A)^\theta$-boundedness]\label{athetabound}
A linear operator $B:D(B)\to X$ is called \emph{$(-A)^\theta$-bounded} if $D((-A)^\theta)\subseteq D(B)$ for some $0\leq \theta <1$, and there exists a constant $d>0$ such that 
\[
\|Bx\|\leq d\,\|(-A)^\theta x\|, \qquad \forall x\in D((-A)^\theta).
\]
\end{Def}
\begin{rk}
Among the three perturbation notions defined above,
$A$-boundedness is the most general one, while $(-A)^\theta$-boundedness is the
most restrictive. The notion of $\theta$-subordination lies between the other two and is the perturbation framework adopted in the present work. Indeed, using the classical interpolation inequality~\cite[Theorem~6.10]{Pazy} and Young's inequality, we obtain that, for every $x\in D(A)$ and every $0<\theta<1$, there exists a constant
$c_\theta>0$ such that for every $\varepsilon>0$
\begin{equation}\label{rel-theta-Abound-relation}
 \|(-A)^\theta x\| 
   \leq c_\theta \|Ax\|^\theta \|x\|^{1 - \theta}
   \leq  \varepsilon\|Ax\|+C(\varepsilon)\|x\|,
\end{equation} where,
$
C(\varepsilon)=
c_\theta(1-\theta)
\left({c_\theta\theta}/{\varepsilon}
\right)^{\frac{\theta}{1-\theta}}.
$ For $\theta=0$, the conclusion follows directly.
Consequently, $(-A)^\theta$-boundedness implies $\theta$-subordination with the same exponent $0\leq\theta<1$, which in turn implies $A$-boundedness. However, the converse statements are not true: $\theta$-subordination does not imply $(-A)^\theta$-boundedness~\cite{Markus-Mat1982}, and it is a stronger property than
$A$-bounded\-ness (see, for example, the perturbations considered in Section~\ref{sec4}). Hence, all the results on the preservation of quasi-compactness
and uniform exponential stability under $\theta$-subordinate perturbations established in Sections~\ref{sec2}-\ref{sec3} apply, in particular, to the class of $(-A)^\theta$-bounded perturbations in the sense of Definition~\ref{athetabound}. On the other hand, in Section~\ref{sec4} we construct an $A$-bounded perturbation which fails to be $\theta$-subordinate and prove the lack of uniform stabilization under such $A$-bounded perturbations. 
\end{rk}
We next recall the definition of quasi-compactness (see \cite[Definition~V.3.4]{EnNa}), which plays a crucial role in the stability analysis developed in this work.
\begin{Def} A $C_0$- semigroup $\left(e^{tA}\right)_{t\geq0}$ on a Banach space $X$ is called \emph{quasi-compact} if 
$$\lim_{t\to\infty}\inf \Big\{\|e^{tA}-K\|_{\mathcal{L}}\,\Big| \, K\in\mathcal{L}(X), \, K \,\text{compact} \Big\}=0.$$
\end{Def} Quasi-compact semigroups are characterized by a negative \emph{essential growth bound} $w_{\text{ess}}$ \cite[Proposition V.3.5]{EnNa}. That is, for a quasi-compact semigroup $\left(e^{tA}\right)_{t\geq0}$, we have $$w_{\text{ess}}:= \inf_{t>0}\frac{1}{t}\,\log\|e^{tA}\|_{\text{ess}}<0,$$where $\|\cdot\|_{\text{ess}}$ is the \emph{essential norm} defined as \begin{equation}
    \|\cdot\|_{\text{ess}}:=\inf\Big\{\|\cdot\,-K\|_{\mathcal{L}}\,\Big| \, K\in\mathcal{L}(X), \, K \,\text{compact} \Big\}
\end{equation}
To formulate the spectral decomposition of a quasi-compact semigroup, we now recall some terminology concerning the spectrum and the resolvent. For an operator $A$, we use $\rho(A)$ and $\sigma(A)$ to denote its resolvent set and spectrum, respectively. In what follows, for $z\in\C$, we denote by $\Re z$ and $\Im z$ its real and imaginary parts, respectively. Given $\lambda\in\rho(A)$, we write $R(\lambda,A):=(\lambda I-A)^{-1}$ for the resolvent operator of $A$ and denote the \emph{spectral bound} of $A$ by $s(A),$ which is defined as \begin{equation}\label{spcbd}
    s(A)= \sup\{\Re\lambda \,\,|\,\,\lambda\in\sigma(A)\}.
\end{equation}  We recall that, for an isolated point $\mu\in\sigma(A)$, the resolvent $R(\lambda,A)$ admits the Laurent series expansion \cite[\S IV.1.17]{EnNa}
$$
R(\lambda,A)=\sum_{n=-\infty}^{\infty}
(\lambda-\mu)^n U_n
$$for $0<|\lambda-\mu|<\delta$ and some sufficiently small $\delta>0$, where the coefficients $U_n$ are bounded operators given by
\begin{equation*}
U_n=\frac{1}{2\pi i}
\int_{\gamma}
\frac{R(\lambda,A)}
     {(\lambda-\mu)^{n+1}}
\,\mathrm{d}\lambda,
\qquad n\in\mathbb{Z},
\end{equation*}where $\gamma$ is a positively oriented Jordan path contained in $\rho(A)$ whose interior contains $\mu$ but no point of $\sigma(A)\setminus\{\mu\}$. The coefficient $U_{-1}$ coincides with the spectral projection $P :=\frac{1}{2\pi i}\int_\gamma R(\lambda,A)\,\mathrm{d}\lambda$
associated with the decomposition
$\sigma(A)=
\{\mu\}
\cup
\bigl(\sigma(A)\setminus\{\mu\}\bigr)
$ of the spectrum of $A$. This coefficient $U_{-1}$ is referred to as the \emph{residue} of $R(\cdot,A)$ at $\mu$. If there exists $k>0$ such that
$$
U_{-k}\neq0,
\qquad
U_{-n}=0
\quad\text{for all }n>k,
$$
then $\mu$ is said to be a \emph{pole} of $R(\cdot,A)$ of order $k$, and we write $$ U_{-k} =\lim_{\lambda\to\mu}(\lambda-\mu)^{k}\,R(\lambda,A).$$ In our analysis, we rely on the spectral decomposition theorem for quasi-compact semigroups~\cite[Theorem~V.3.7]{EnNa}, which states the following.
\begin{thm}\label{thm-quasicompact-spectral-decomp}
Let $\left(e^{tA}\right)_{t\geq0}$ be a quasi-compact strongly continuous
semigroup with generator $A$ on a Banach space $X$. Then the following holds:
\begin{enumerate}
\item[(i)] The set
$
\left\{\mu\in\sigma(A): \Re\mu\geq0\right\}
$
is finite, possibly empty, and consists of poles of $R(\cdot,A)$ of finite
algebraic multiplicity.
\item[(ii)] Denote by
$\mu_1,\ldots,\mu_m$ the poles of $R(\cdot,A)$,
their residues by
$U_1,\ldots,U_m,$
and the corresponding orders of the poles by $k_1,\ldots,k_m $.
Then, for every $t\geq0$, we have
\begin{equation}\label{spectdecomp}
   e^{tA}=T_1(t)+T_2(t)+\cdots+T_m(t)+R(t) 
\end{equation}
where
$$T_n(t)=e^{\mu_n t}
\sum_{j=0}^{k_n-1}
\frac{t^j}{j!}(A-\mu_n)^j\,U_n,
\qquad
t\geq0,\quad 1\leq n\leq m,
$$
and
$$
\|R(t)\|_{\mathcal L(X)}\leq
\widehat M e^{-\widehat w t},
\qquad t\geq0,
$$
for some constants $\widehat w>0$ and $\widehat M\geq1$.
\end{enumerate}
\end{thm}}

\subsection{Background review}  {Early foundational work in the study of stabilization of PDEs was carried out by Russell (see \cite{Russell-1969, Russell-1975}), who established decay estimates and control-theoretic methods for hyperbolic equations, highlighting fundamental limitations on achievable decay rates. The general abstract linear stabilization problem in Hilbert spaces was subsequently investigated in \cite{Slemrod1972}.} Furthermore, in the Hilbert space setting, Gibson \cite{Gibson-1980} proved that if $\left(e^{tA}\right)_{t\geq0}$ is a
strongly stable contraction semigroup and $B\in\mathcal L(H)$ is compact, then
the uniform exponential stability of $\left(e^{t(A+B)}\right)_{t\geq0}$ implies
the uniform exponential stability of the original semigroup $\left(e^{tA}\right)_{t\geq0}$. In other words, compact perturbations (as well as compact feedback gain operators, in the sense of~\eqref{closedSys}) cannot produce uniform exponential stability unless the unperturbed
semigroup is already uniformly exponentially stable. This result also indicates
the permanence of exponential stability under compact
perturbations of the generator of a uniformly stable semigroup of bounded
operators. This result was later extended by Triggiani
\cite{Triggiani1988, Triggiani-1989}, who removed the contraction assumption on~$\left(e^{tA}\right)_{t\geq0}$ and proved the permanence of
exponential stability under compact perturbations. Alabau-Boussouira and
Cannarsa~\cite{AB-C} considered exponentially stable semigroups on
reflexive Banach spaces and replaced the compactness of the perturbation
operator $B$ itself by the weaker assumption that $B$ is bounded and $Be^{tA}$ is compact for every
$t>0$. This result was later extended in~\cite{SLRRS-PC-RG} to a class of unbounded
perturbations~$B$. More precisely, for a class of $\theta$-subordinate operators~$B$, in~\cite{SLRRS-PC-RG} we proved the permanence of exponential stability for generators of exponentially
stable analytic semigroups. Applications of the abstract results to uniformly parabolic equations, degenerate/singular parabolic equations, coupled plate dynamics, generalized coupled systems of Kirchhoff–Love plates with membrane-like electrical networks, and interior degenerate/singular parabolic equations were presented in \cite[\S 3]{SLRRS-PC-RG} and \cite{Shri-Cann-R}, respectively.

A substantial amount of work has addressed the problem of linear stabilization in infinite-dimensional Hilbert spaces (see, for instance, \cite{Russell-1969, Russell-1975, Slemrod1972, Slemrod-1974}). The contributions of Gibson \cite{Gibson-1980} and Triggiani \cite{Triggiani1988, Triggiani-1989} are also framed in Hilbert spaces, whereas the stabilization results of \cite{AB-C, SLRRS-PC-RG} were established in reflexive Banach spaces. The aim of the present paper is to extend the result of \cite{SLRRS-PC-RG} to the setting of general Banach spaces. {This extension is not a straightforward consequence of the arguments developed in \cite{SLRRS-PC-RG}, since many compactness results in reflexive Banach spaces are no longer available in an arbitrary Banach space. Instead, in this paper we provide a different proof in arbitrary Banach spaces, relying on completely different techniques compared to the methods used in~\cite{SLRRS-PC-RG}. Here, we develop a perturbation theory for quasi-compactness of semigroups on general Banach spaces, and then show how this can be exploited in proving permanence of exponential stability for the perturbed semigroup.}

\paragraph{Approaches and novelties.}
The proof of \cite[Theorem 1.1]{AB-C} is based on the variation of constants formula applied to \eqref{closedSys}, which reads as
$$
e^{t(A+B)}x=
e^{tA}x
+
\int_0^t e^{(t-s)(A+B)}Be^{sA}x \,\mathrm{d}s ,
$$
and the reflexivity of the underlying Banach space plays an essential role. Indeed, the argument relies on extracting a weakly convergent subsequence from a bounded sequence in the unit ball and then using the compactness of the operator $Be^{tA}$ to obtain strong convergence. Such an approach is no longer available in general Banach spaces, since bounded subsets of a Banach space need not be weakly compact. In the present paper, we overcome this difficulty by avoiding weak compactness arguments altogether. Instead, we prove directly that the perturbation term arising from the variation of constants formula is compact for every positive time.  To do so, we first show that quasi-compactness of an analytic semigroup is preserved under $\theta$-subordinate perturbations, provided that the perturbation is compact along the trajectories of the unperturbed semigroup. 
Quasi-compact semigroups are closely connected to stability properties. The spectral decomposition theorem (Theorem \ref{thm-quasicompact-spectral-decomp}) for quasi-compact semigroups tells us that the long-time dynamics of such semigroups roughly split into a compact finite-dimensional spectral part and an exponentially decaying remainder. 
As a consequence of this theorem, every exponentially stable semigroup is quasi-compact; however, the converse is not true in general. {A sufficient condition under which the converse statement holds is given as a corollary to Theorem \ref{thm-quasicompact-spectral-decomp} in \cite[Corollary V.4.7]{EnNa2}. It states that a quasi-compact strongly continuous semigroup with generator $\widetilde{A}$ is uniformly exponentially stable if and only if $s(\widetilde{A})<0,$ where $s(\widetilde{A})$ denotes the spectral bound defined in \eqref{spcbd}. We prove that quasi-compactness together with strong stability implies uniform exponential stability by showing that strong stability rules out the finite-dimensional spectral modes associated with spectral values in the closed right half-plane, while the remaining part of the semigroup decays exponentially by Theorem \ref{thm-quasicompact-spectral-decomp}.}

\paragraph{Organization of the paper.} The remainder of the paper is organized as follows. In Section~\ref{sec2}, we  provide sufficient conditions to ensure the
{permanence} of quasi-compactness for analytic semigroups under $\theta$-subordinate perturbations which
are compact along trajectories. We also derive corresponding variants for the class of Crandall--Pazy semigroups and immediately differentiable semigroups. In
Section~\ref{sec3} we prove the main uniform stabilization theorem on general Banach spaces, and apply it to a perturbed diffusion operator in Section~\ref{app}. In Section~\ref{sec4}, we illustrate the necessity of the assumptions in the main theorem by constructing counterexamples that show the
lack of uniform stabilization for general $A$-bounded perturbations or dropping the assumption on the compactness-along-trajectories of the perturbation.

\section{Permanence of Quasi-Compactness}\label{sec2}
In \cite[Proposition V.4.9]{EnNa2}, the authors proved that if the generator of a strongly continuous quasi-compact semigroup is perturbed by a compact operator, then the perturbed operator generates a quasi-compact semigroup again. We extend this result to accommodate unbounded perturbations of $\theta$-subordinate type under certain conditions on the perturbation and a regularity assumption on the semigroup. Precisely, in the sequel, we prove that quasi-compactness of an analytic semigroup is preserved under $\theta$-subordinate perturbations of its generator, provided the perturbations are compact along trajectories.
\begin{Lem}\label{quasi-compact-preservatn}
    Let $A$ be the generator of a strongly continuous quasi-compact analytic semigroup $\left(e^{tA}\right)_{t\geq0}$ of bounded linear operators on the Banach space $X$. Consider a linear operator $B:D(B)\to X$ such that 

    \begin{itemize}
        \item[\hypertarget{H1}{\emph{\textbf{H1)}}}] $B$ is $\theta$-subordinate to $A$ in the sense of Definition \ref{theta-sub}.

        \item[\hypertarget{H2}{\emph{\textbf{H2)}}}] $\forall\,t>0$, $Be^{tA}$ is compact.  

    \end{itemize}
    
     Then, the semigroup $\left(e^{t(A+B)}\right)_{t\geq0}$ generated by $A+B$ is quasi-compact.
\end{Lem}

    \begin{proof}
First note that, since $A$ generates an analytic $C_0$-semigroup, there exist constants
$M_0\geq1$ and $\omega_0\geq 0 $ such that
\begin{align}
    \|e^{sA}\|_{\mathcal L}
    &\leq M_0 e^{\omega_0 s},
    \qquad s\geq 0.
    \label{C-0-A}
\end{align}
Therefore, $A-\omega_0I$ generates a bounded analytic semigroup. Using the result for bounded analytic semigroups in \cite[Theorem II.4.6]{EnNa},  there exists a constant $C_0>0$ such that
$$
\|(A-\omega_0I)e^{s(A-\omega_0I)}\|_{\mathcal L}
\leq \frac{C_0}{s},
\qquad s>0.
$$ Hence, there exists $M_1:= \max\{C_0,\omega_0M_0\}$ such that 
\begin{align}
\|Ae^{sA}\|_{\mathcal L}
&\leq
\|(A-\omega_0I)e^{sA}\|_{\mathcal L}
+
\omega_0\|e^{sA}\|_{\mathcal L}
\nonumber\\
&=
e^{\omega_0s}
\|(A-\omega_0I)e^{s(A-\omega_0I)}\|_{\mathcal L}
+
\omega_0\|e^{sA}\|_{\mathcal L}
\nonumber\\
&\leq
e^{\omega_0s}
\left(
\frac{C}{s}+\omega_0M_0
\right)
\leq
M_1e^{\omega_0s}
\left(1+\frac{1}{s}\right),
\qquad s>0.
\label{Analytic-A}
\end{align} Therefore, by hypothesis
\hypref{H1}{H1)},
\begin{align*}
\|Be^{sA}x\|
&\leq
c\|Ae^{sA}x\|^\theta
\|e^{sA}x\|^{1-\theta}\leq
cM_1^\theta M_0^{1-\theta}
e^{\omega_0s}
\left(1+\frac1s\right)^\theta
\|x\|.
\end{align*}
Hence, for every $t_0>0$,
\begin{align*}
\int_0^{t_0}\|Be^{sA}x\|\;\mathrm{d}s
&\leq
cM_1^\theta M_0^{1-\theta}
e^{\omega_0t_0}(1+t_0)^\theta
\int_0^{t_0}s^{-\theta}\;\mathrm{d}s\,\|x\|\\
&=
\frac{cM_1^\theta M_0^{1-\theta}}
{1-\theta}
e^{\omega_0t_0}(1+t_0)^\theta
t_0^{1-\theta}\|x\|.
\end{align*}
Since $0\le \theta<1$, the coefficient on the right-hand side converges to
zero as $t_0\to0^+$. Thus, choosing $t_0>0$ sufficiently small, there
exists $q<1$ such that
$$
\int_0^{t_0}\|Be^{sA}x\|\;\mathrm{d}s
\leq q\|x\|,
\qquad x\in D(A).
$$
Hence, by \cite[Corollary III.3.16]{EnNa}, the operator $A+B$ with domain $D(A+B)=D(A)$ generates a strongly continuous semigroup, and by the variation of constants formula applied to \eqref{closedSys}, for every $x\in X$ and every $t\geq 0$,
we have $$e^{t(A+B)}x=e^{tA}x+\Lambda_t x,$$
where $$\Lambda_t x:=\int_0^t e^{(t-s)(A+B)}Be^{sA}x\,ds .$$
Since $e^{t(A+B)}-e^{tA}=\Lambda_t ,$ in order to prove that $\left(e^{t(A+B)}\right)_{t\geq0}$ is
quasi-compact, it is enough to prove that $\Lambda_t$ is compact for every
$t>0$. Indeed, since $(e^{tA})_{t\geq0}$ is quasi-compact, thanks to \cite[Proposition V.3.5]{EnNa} there exists $t_0>0$
such that $\|e^{t_0A}\|_{\mathrm{ess}}<1.$ 
Because compact perturbations do not change the essential norm, we obtain
$$
\|e^{t_0(A+B)}\|_{\mathrm{ess}}
=
\|e^{t_0A}+\Lambda_{t_0}\|_{\mathrm{ess}}
=
\|e^{t_0A}\|_{\mathrm{ess}}
<1,
$$
and thus $(e^{t(A+B)})_{t\geq0}$ is quasi-compact. In the following steps, we prove that $\Lambda_t$ is compact for every $t>0.$

\hypertarget{stp1}{\textbf{Step 1.}} 
Fix $t>0$. For $\varepsilon \in(0,t)$, define $$\Lambda_{t,\varepsilon }x
:=
\int_{\varepsilon }^{t} e^{(t-s)(A+B)}Be^{sA}x\,ds \,,\qquad x\in X\, .$$
We first prove that $\Lambda_{t,\varepsilon }$ is compact.  Note that the map  $s\mapsto Be^{sA}$ is norm-continuous from $[\varepsilon ,t]$ into $\mathcal{L}(X)$. Indeed, since $A$ generates an analytic semigroup, for every $\varepsilon >0$, the maps $s\mapsto e^{sA}$ and $s\mapsto Ae^{sA}$ are norm-continuous  for any $s\in [\varepsilon,t]$~\cite[\S II.2.c]{EnNa}. Hence, for $r,s\in[\varepsilon ,t]$ and $x\in X$, by \hypref{H1}{H1)}, we get
$$
\left\|B\left(e^{rA}-e^{sA}\right)x\right\|
\leq
c\left\|A\left(e^{rA}-e^{sA}\right)x\right\|^{\theta}
\left\|\left(e^{rA}-e^{sA}\right)x\right\|^{1-\theta}.
$$
Therefore,
\begin{equation}
    \label{Be^{sA}}
   \left\|B\left(e^{rA}-e^{sA}\right)\right\|_{\mathcal L}
\leq
c
\left\|A\left(e^{rA}-e^{sA}\right)\right\|_{\mathcal L}^{\theta}
\left\|e^{rA}-e^{sA}\right\|_{\mathcal L}^{1-\theta}
\rightarrow0\,\,\text{as}\,\, r\rightarrow s. 
\end{equation}

\hypertarget{stp2}{\textbf{Step 2.}} Let
$$\mathbb U_{[\varepsilon ,t]}
:=
\bigcup_{s\in[\varepsilon ,t]}Be^{sA}(\mathbb B_X),
$$
where
$
\mathbb B_X:=\{x\in X:\|x\|\leq1\}.
$
We claim that $\mathbb U_{[\varepsilon ,t]}$ is relatively compact in~$X$. 

Since $s\mapsto Be^{sA}$ is norm-continuous on the compact interval
$[\varepsilon ,t]$, the set
$$
\mathbb T_{[\varepsilon ,t]}:=\{Be^{sA}:s\in[\varepsilon ,t]\}
$$
is compact in $\mathcal L(X)$. Moreover, by hypothesis \hypref{H2}{H2)}, $Be^{sA}$ is compact
for every $s>0$. Hence $\mathbb T_{[\varepsilon ,t]}\subset \mathcal K(X).$

Let $\eta>0$. By compactness of $\mathbb T_{[\varepsilon ,t]}$ in
$\mathcal L(X)$, there exist $s_1,\ldots,s_N\in[\varepsilon ,t]$ such that, setting $K_j:=Be^{s_jA},\, j=1,\ldots,N,$
we have
$$
\mathbb T_{[\varepsilon ,t]}
\subset
\bigcup_{j=1}^N
\mathbb B_{\mathcal L(X)}(K_j,\eta/2).
$$
That is, for every $s\in[\varepsilon ,t]$, there exists
$j\in\{1,\ldots,N\}$ such that
\begin{equation}\label{K-compact}
    \|Be^{sA}-K_j\|_{\mathcal L}<\eta/2.
\end{equation}

Since each $K_j=Be^{s_jA}$ is compact, each set $K_j(\mathbb B_X)$ is relatively compact in~$X$. Therefore, each $K_j(\mathbb B_X)$ is totally bounded. Hence, for each $j=1,\ldots,N$, there exists a finite set
$L_j\subset X$ such that $L_j$ is an $\eta/2$-net for $K_j(\mathbb B_X)$; that
is, for every $x\in\mathbb B_X$, there exists $z\in L_j$ such that
\begin{equation}\label{zineta/2net}
    \|K_jx-z\|<\eta/2.
\end{equation}

Now define the finite set $L:=\bigcup_{j=1}^N L_j.$ We observe that $L$ is an $\eta$-net for $\mathbb U_{[\varepsilon ,t]}$. Indeed, let $y\in\mathbb U_{[\varepsilon ,t]}$. Then there exist
$s\in[\varepsilon ,t]$ and $x\in\mathbb B_X$ such that $y=Be^{sA}x.$
Choose $j\in\{1,\ldots,N\}$ such that \eqref{K-compact} holds. Then \eqref{zineta/2net} ensures that there exists~$z\in L_j\subset L$ such that
$$
\begin{aligned}
\|y-z\|
&=
\|Be^{sA}x-z\| \leq
\|Be^{sA}x-K_jx\|+\|K_jx-z\|  \\
&\leq
\|Be^{sA}-K_j\|_{\mathcal L}\,\|x\|+\|K_jx-z\| <\eta. 
\end{aligned}
$$
Thus, $L$ is a finite $\eta$-net for $\mathbb U_{[\varepsilon ,t]}$. Since $\eta>0$ is arbitrary, $\mathbb U_{[\varepsilon ,t]}$ is totally bounded in
$X$. Since $X$ is a Banach space, total boundedness of
$\mathbb U_{[\varepsilon ,t]}$ implies that $\mathbb U_{[\varepsilon ,t]}$ is relatively compact in $X$ \cite[Theorem 3.17.13]{NaylorSell1982}.

\vspace{2mm}

\hypertarget{stp3}{\textbf{Step 3.}} We now show that the map
$
(\tau,y)\mapsto e^{\tau(A+B)}y
$
is jointly continuous from $[0,t]\times X$ into $X$. Let
$\tau_n\to \tau$ in $[0,t]$ and let $y_n\to y$ in $X$. Then
$$
\begin{aligned}
\|e^{\tau_n(A+B)}y_n-e^{\tau(A+B)}y\|
&\leq
\|e^{\tau_n(A+B)}(y_n-y)\|
+
\|e^{\tau_n(A+B)}y-e^{\tau(A+B)}y\| \\
&\leq
\|e^{\tau_n(A+B)}\|\,\|y_n-y\|
+
\|e^{\tau_n(A+B)}y-e^{\tau(A+B)}y\|.
\end{aligned}
$$
Since $(e^{\tau(A+B)})_{\tau\geq0}$ is a $C_0$-semigroup, by
\cite[Theorem 1.2.2]{Pazy}, there exist $\tilde M\geq1$ and
$\tilde\omega\geq 0$ such that $\|e^{\tau(A+B)}\|\leq\tilde Me^{\tilde \omega \tau},
    \, \tau\geq0.$
In particular, for any $\tau_n\in[0,t]$, $\|e^{\tau_n(A+B)}\|
    \leq \tilde Me^{\tilde \omega t}.$ Therefore $\|e^{\tau_n(A+B)}\|\|y_n-y\|\to0.$
Moreover, by the strong continuity of the semigroup, $ \|e^{\tau_n(A+B)}y-e^{\tau(A+B)}y\|\to0.$
Consequently,
\[
    \|e^{\tau_n(A+B)}y_n-e^{\tau(A+B)}y\|\to0.
\]
Hence, the map
$
(\tau,y)\mapsto e^{\tau(A+B)}y
$
is jointly continuous from $[0,t]\times X$ into $X$.

Recall from Step \hyperlink{stp2}{{\textbf{2}}} that
$\overline{\mathbb U}_{[\varepsilon ,t]}$ is compact in $X$. Therefore,
$[\varepsilon ,t]\times \overline{\mathbb U}_{[\varepsilon ,t]}$
is compact. Since the map
$
(s,y)\mapsto e^{(t-s)(A+B)}y
$
is continuous from
$[\varepsilon ,t]\times \overline{\mathbb U}_{[\varepsilon ,t]}$
into $X$, its image
$$
\mathbb M_{[\varepsilon ,t]}
:=
\left\{
e^{(t-s)(A+B)}y\,:\,s\in[\varepsilon ,t],\ y\in\overline{\mathbb U}_{[\varepsilon ,t]}
\right\}
$$
is compact in $X$. 

\vspace{2mm}
\hypertarget{stp4}{\textbf{Step 4.}} We now prove that  $\Lambda_{t,\varepsilon}(\mathbb B_X)$ is relatively compact. Let $x\in\mathbb B_X$. For every $s\in[\varepsilon ,t]$, we have $Be^{sA}x\in\mathbb U_{[\varepsilon ,t]}.$ Therefore, $
e^{(t-s)(A+B)}Be^{sA}x\in\mathbb M_{[\varepsilon ,t]}
$. 
Hence, the integrand in
$$
\Lambda_{t,\varepsilon }x
=
\int_\varepsilon ^t e^{(t-s)(A+B)}Be^{sA}x\,ds
$$
takes values in the compact set $\mathbb M_{[\varepsilon ,t]}$. Observe that  $\Lambda_{t,\varepsilon }x$ is the norm limit of Riemann sums of the form
$$
\sum_{j=1}^n \Delta s_j\,e^{(t-s_j)(A+B)}Be^{s_jA}x\,.
$$
Dividing by $t-\varepsilon $, we get a convex combination of elements of $\mathbb M_{[\varepsilon ,t]}$.
Therefore,
$$
\frac{1}{t-\varepsilon }\Lambda_{t,\varepsilon }x
\in
\overline{\operatorname{co}}\left(\mathbb M_{[\varepsilon ,t]}\right).
$$
Since $\mathbb M_{[\varepsilon ,t]}$ is  compact, its closed convex hull
is compact in the Banach space~$X$ by Mazur's Theorem \cite[Theorem IV.4.8]{Conway1990}. Hence,
$$
\Lambda_{t,\varepsilon }(\mathbb B_X)
\subset
(t-\varepsilon )\;\overline{\operatorname{co}}\left(\mathbb M_{[\varepsilon ,t]}\right)
$$
is relatively compact in $X$. Thus, $\Lambda_{t,\varepsilon }$ is compact (see \cite[Definition 3.1]{Conway1990}).

\vspace{2mm}
\hypertarget{stp5}{\textbf{Step 5.}}
Since $B$ is a $\theta$-subordinate perturbation of $A$, it follows from
\eqref{rel-theta-Abound-relation} and Definition~\ref{Abound} that $B$ is
$A$-bounded with $A$-bound zero.
Therefore, by the perturbation theorem for analytic semigroups \cite[Theorem 3.2.1]{Pazy}, $A+B$
also generates an analytic $C_0$-semigroup. Hence, 
there exist constants
$M_2\geq 1$ and $\omega_1\geq0$ such that
\begin{equation}\label{C-0-A+B}
    \|e^{(t-s)(A+B)}\|_{\mathcal L}
    \leq M_2e^{\omega_1(t-s)},
    \qquad 0\leq s\leq t.
\end{equation}

Using \hypref{H1}{H1)}, \eqref{C-0-A}, \eqref{Analytic-A} and \eqref{C-0-A+B}, we obtain, for any $0<s\leq t$,
\begin{align}
    \|e^{(t-s)(A+B)}Be^{sA}x\|
    &\leq
    \|e^{(t-s)(A+B)}\|_{\mathcal L}\,
    \|Be^{sA}\|_{\mathcal L}\,\|x\| \nonumber\\
    &\leq
    cM_2e^{\omega_1(t-s)} 
    \left[
        M_1e^{\omega_0s}\left(1+\frac1s\right)
    \right]^\theta
    \left[
        M_0e^{\omega_0s}
    \right]^{1-\theta}
    \|x\| \nonumber\\
    &=
    cM_2 M_1^\theta M_0^{1-\theta}
    e^{\omega_1(t-s)+\omega_0s}
    \left(1+\frac1s\right)^\theta
    \|x\| \nonumber\\
    &\leq C(t)s^{-\theta}\|x\|\,,\label{B-T-estimate-2}
\end{align}
where $ C(t)
    :=
    cM_0^{1-\theta}M_1^\theta M_2
    e^{\max\{\omega_0,\omega_1\} t}(1+t)^\theta.$ Consequently, for $\|x\|\leq 1$,
\begin{align}
    \|(\Lambda_t-\Lambda_{t,\varepsilon })x\|
    &=
    \left\|
    \int_0^\varepsilon  e^{(t-s)(A+B)}Be^{sA}x\,ds
    \right\| \nonumber\\
    &\leq
    \int_0^\varepsilon 
    \|e^{(t-s)(A+B)}Be^{sA}\|_{\mathcal L}
    \,\|x\|\,ds \nonumber\\
    &\leq
    C(t)\int_0^\varepsilon s^{-\theta}\,ds 
    =
    \frac{C(t)}{1-\theta}\,\varepsilon^{1-\theta}.\label{limitpasLamb}
\end{align}
Since $\theta<1$, it follows that $\varepsilon ^{1-\theta}\to 0$  as $\varepsilon \downarrow 0.$
Therefore,
\[
    \|\Lambda_t-\Lambda_{t,\varepsilon }\|_{\mathcal L}
    \to 0
    \qquad\text{as}\qquad
    \varepsilon \downarrow 0.
\] 

Since each $\Lambda_{t,\varepsilon }$ is compact and the space of compact
operators is closed in $\mathcal L(X)$ \cite[Theorem 5.24.8]{NaylorSell1982}, it follows that $\Lambda_t$ is compact. 
\end{proof}
The preservation of quasi-compactness established above in fact highlights a stronger invariance property. Indeed, the proof of Lemma \ref{quasi-compact-preservatn} shows that, for every $t>0$, $e^{t(A+B)}-e^{tA}=\Lambda_t\in\mathcal K(X).$ Hence \(e^{t(A+B)}\) and \(e^{tA}\) determine the same element of the Calkin algebra $\mathfrak C(X):=\mathcal L(X)/\mathcal K(X).$ In other words, if $\pi:\mathcal L(X)\longrightarrow \mathcal L(X)/\mathcal K(X)$ denotes the canonical quotient map onto $\mathfrak C(X)$, then $$\pi\,\left(e^{t(A+B)}\right)= \pi\,\left(e^{tA}\right),\qquad t\geq0.$$ Thus, for every $t>0$, it follows that $\|e^{t(A+B)}\|_{\text{ess}}=\|e^{tA}\|_{\text{ess}},$ which implies $$\omega_{\mathrm{ess}}(A+B) =\inf_{t>0}\frac1{t}\log\|e^{t(A+B)}\|_{\mathrm{ess}}=\inf_{t>0}\frac1{t}\log\|e^{tA}\|_{\mathrm{ess}}=\omega_{\mathrm{ess}}(A).$$ 
Therefore, under hypotheses
\hypref{H1}{H1)} and
\hypref{H2}{H2)}, the essential growth bound and the essential norm are preserved for every $t>0$.
\begin{rk}
  Note that the proof of Lemma \ref{quasi-compact-preservatn} suggests that
  the analyticity assumption can be relaxed to the class of Crandall-Pazy
  semigroups, at the cost of allowing for a smaller class of
  $\theta$-subordinate perturbations $B$. Indeed, the analyticity assumption is
  mainly used in the estimate \eqref{Analytic-A} to obtain
  \eqref{B-T-estimate-2}, and so to conclude the convergence in
  \eqref{limitpasLamb}. 
\end{rk}

Recall the definition of the class of \emph{Crandall-Pazy semigroups with parameter
$\alpha\in(0,1]$} (see \cite{Crandall-Pazy}, \cite[\S 1]{Wakaiki}).

\begin{Def}
Let $(T(t))_{t\geq 0}$ be a $C_0$-semigroup on a Banach space $X$ with generator $A$. The semigroup $(T(t))_{t\geq 0}$ is said to be in the \emph{Crandall-Pazy class} with parameter
$\alpha\in(0,1]$ if $(T(t))_{t\geq0}$ is immediately differentiable and
\begin{equation}\label{eqCP}
\|A T(t)\|_{\mathcal{L}} = O\left(\frac{1}{t^{1/\alpha}}\right)
\qquad \text{as } t\to 0^{+}.
\end{equation}
\end{Def}

We now state the corresponding quasi-compactness preservation result.

\begin{Cor}\label{crandall-pazy-Quasi-Compact}
Let $A$ be the generator of a quasi-compact $C_0$-semigroup
$\left(e^{tA}\right)_{t\geq0}$ on a Banach space $X$. Assume that $\left(e^{tA}\right)_{t\geq0}$ belongs
to the Crandall-Pazy class with parameter $\alpha\in(0,1]$. Let
$B:D(B)\to X$ be a linear operator such that $B$ is $\theta$-subordinate to $A$ in the sense of
    Definition \ref{theta-sub}, with $0<\theta<\alpha\leq1$ and 
$Be^{tA}$ is compact for all $t>0.$
Then $\left(e^{t(A+B)}\right)_{t\geq0}$ is quasi-compact.
\end{Cor}

\begin{proof}
Since $\left(e^{tA}\right)_{t\geq0}$ belongs to the Crandall--Pazy class
with parameter $\alpha\in(0,1]$, there exist $\delta>0$ and
$C_\delta>0$ such that
\[
\|Ae^{sA}\|_{\mathcal L}
\leq
C_\delta s^{-1/\alpha},
\qquad 0<s\leq\delta.
\]
Moreover, since $\left(e^{tA}\right)_{t\geq0}$ is a $C_0$-semigroup, it is exponentially bounded (see~\cite[Theorem~1.2.2]{Pazy}). Hence, for every $\delta>0$, there exists $M_\delta>0$ such that
$$
\|e^{sA}\|_{\mathcal L}\leq M_\delta,
\qquad 0\leq s\leq\delta.
$$
Therefore, by the $\theta$-subordination of $B$ to $A$, for every
$x\in X$ and $0<s\leq\delta$, we have
$$
\|Be^{sA}x\|
\leq
c\|Ae^{sA}x\|^\theta
\|e^{sA}x\|^{1-\theta} 
\leq
cC_\delta^\theta M_\delta^{1-\theta}
s^{-\theta/\alpha}\|x\|.
$$

Since $\theta<\alpha$, we have $\theta/\alpha<1$, and thus, for every $t_0\in(0,\delta]$,
$$
\int_0^{t_0}\|Be^{sA}x\|\,ds
\leq
\frac{cC_\delta^\theta M_\delta^{1-\theta}}
{1-\theta/\alpha}
t_0^{\,1-\theta/\alpha}\|x\|.
$$
The coefficient on the right-hand side tends to zero as
$t_0\to0^+$. Thus, we can choose $t_0>0$ sufficiently small so that there exists
$q<1$ such that
$$
\int_0^{t_0}\|Be^{sA}x\|\,ds
\leq q\|x\|,
\qquad x\in D(A).
$$It then follows from \cite[Corollary~III.3.16]{EnNa} that $A+B,$ with domain $D(A+B)=D(A),$ generates a strongly continuous semigroup
$\left(e^{t(A+B)}\right)_{t\geq0}$, and we have
$$
e^{t(A+B)}x
=
e^{tA}x+
\int_0^t
e^{(t-s)(A+B)}Be^{sA}x\,ds,
\qquad x\in X,\quad t\geq0.
$$
The rest of the proof follows the same argument as in Lemma \ref{quasi-compact-preservatn}, except for Steps~\hyperlink{stp1}{{\textbf{1}}} and \hyperlink{stp5}{{\textbf{5}}}. More precisely, Step \hyperlink{stp1}{{\textbf{1}}} is modified to prove the norm-continuity of the map $s\mapsto Be^{sA}$ on $[\varepsilon ,t]$ as follows. Fix $\varepsilon\in (0,t)$. Note that, since $\left(e^{tA}\right)_{t\geq0}$ is immediately differentiable, the map $s\mapsto e^{sA}$ is norm-continuous on~$[\varepsilon ,t]$ (see~\cite[\S II.2.c]{EnNa}). We now prove that the map $s\mapsto Ae^{sA}$ is norm-continuous on $[\varepsilon ,t]$. Let $r,s\in[\varepsilon ,t]$ and assume,
without loss of generality, that $r>s>0$. Then
$$
Ae^{rA}-Ae^{sA}
=
Ae^{\frac{s}{2}A}
\left(e^{(r-s)A}-I\right)e^{\frac{s}{2}A}.
$$
Hence, we have
$$
\begin{aligned}
\|Ae^{rA}-Ae^{sA}\|_{\mathcal L}
&\leq
\left\|Ae^{\frac{s}{2}A}\right\|_{\mathcal L}
\left\|
\left(e^{(r-s)A}-I\right)e^{\frac{s}{2}A}
\right\|_{\mathcal L}  \\
&\leq
\left\|Ae^{\frac{s}{2}A}\right\|_{\mathcal L}
\int_0^{r-s}
\|e^{\tau A}\|_{\mathcal L}
\left\|Ae^{\frac{s}{2}A}\right\|_{\mathcal L}
\,d\tau .
\end{aligned}
$$
By the Crandall--Pazy estimate \eqref{eqCP}, the family
$
\left\{Ae^{\frac{s}{2}A}:s\in[\varepsilon ,t]\right\}
$
is bounded in $\mathcal L(X)$. Moreover, using \eqref{C-0-A}, there exists a
constant $C_{\varepsilon ,t}>0$ such that
$$
\|Ae^{rA}-Ae^{sA}\|_{\mathcal L}
\leq
C_{\varepsilon ,t}|r-s|.
$$
Consequently,
$$
\|Ae^{rA}-Ae^{sA}\|_{\mathcal L}\to0
\qquad\text{as } r\to s.
$$
Hence, $s\mapsto Ae^{sA}$ is norm-continuous on $[\varepsilon ,t]$. Then, the norm continuity of the map \(s\mapsto Be^{sA}\) on \([\varepsilon,t]\) follows as in \eqref{Be^{sA}} of Step \hyperlink{stp1}{{\textbf{1}}}.

It remains to prove Step \hyperlink{stp5}{{\textbf{5}}}. 
Using the $\theta$-subordination of $B$ to $A$, together with~\eqref{C-0-A},~\eqref{C-0-A+B} and~\eqref{eqCP}, we obtain, for every $x\in X$ and $0<s\leq t$,
$$
\|e^{(t-s)(A+B)}Be^{sA}x\|
\leq
C(t)s^{-\theta/\alpha}\|x\|.
$$
Since $\theta<\alpha$, we have $\theta/\alpha<1$, and hence
$s^{-\theta/\alpha}$ is integrable near $s=0$. Therefore
$$
\|\Lambda_t-\Lambda_{t,\varepsilon }\|_{\mathcal L}
\leq
C(t)\int_0^\varepsilon  s^{-\theta/\alpha}\,ds
\to0
\qquad\text{as }\varepsilon \downarrow0.
$$
The remaining part of the proof follows exactly the same arguments as in Lemma~\ref{quasi-compact-preservatn}. Hence, $\Lambda_t$ is compact for every $t>0$, and therefore
$\left(e^{t(A+B)}\right)_{t\geq0}$ is quasi-compact.
\end{proof}
In fact, for bounded perturbations $B$, the Crandall-Pazy estimate~\eqref{eqCP} can be removed while retaining immediate differentiability of the semigroup generated by $A$. More precisely, it follows directly from the proof of Lemma \ref{quasi-compact-preservatn} that the preservation of quasi-compactness remains valid for immediately differentiable $C_0$-semigroups, provided the bounded perturbation satisfies
assumption \hypref{H2}{H2)}.
\begin{Cor}\label{Cor2}
Let $A$ be the generator of an immediately differentiable quasi-compact $C_0$-semigroup
$\left(e^{tA}\right)_{t\geq0}$ on a Banach space $X$. Let $B\in\mathcal L(X)$ be a bounded linear operator such that $Be^{tA}$ is compact for all $t>0$.
Then $\left(e^{t(A+B)}\right)_{t\geq0}$ is quasi-compact.
\end{Cor}

\begin{proof}
Note that, since the semigroup $\left(e^{tA}\right)_{t\geq0}$ is immediately differentiable, the map $s\mapsto e^{sA}$ is norm continuous on $[\varepsilon ,t]$ (see \cite[\S II.2.c]{EnNa}). This, together with the boundedness of $B$, implies
$$
\|Be^{rA}-Be^{sA}\|_{\mathcal L}
\leq
\|B\|_{\mathcal L}\,\|e^{rA}-e^{sA}\|_{\mathcal L}
\to 0
\quad \text{as } r\to s,\ \forall\, r,s\in[\varepsilon ,t].
$$
Thus, the map $s\mapsto Be^{sA}$ is norm continuous on $[\varepsilon ,t]$.

We now replace Step \hyperlink{stp5}{{\textbf{5}}} of the proof of Lemma \ref{quasi-compact-preservatn} by the following argument. Exploiting the boundedness of $B$ and estimates \eqref{C-0-A} and \eqref{C-0-A+B}, we obtain
$$
\|\Lambda_t-\Lambda_{t,\varepsilon }\|_{\mathcal L}
\leq
\int_0^\varepsilon 
\left\|e^{(t-s)(A+B)}Be^{sA}\right\|_{\mathcal L}\,ds
\to 0
\qquad \text{as } \varepsilon \downarrow0.
$$
The rest of the proof follows exactly as in Lemma \ref{quasi-compact-preservatn}.
\end{proof}

\section{Uniform stabilization under $\theta$-subordinate perturbations of generators}\label{sec3}
We now state our main stabilization result on a general Banach space under certain conditions on the perturbation $B$ and a regularity assumption on the semigroup generated by $A$. We exploit the fact that quasi-compactness of the semigroup implies uniform exponential stability of the induced semigroup in the Calkin algebra $\mathfrak{C}(X)$ equipped with quotient norm $\|\cdot\|_{\text{ess}}$. We show that quasi-compactness, combined with
strong stability of the perturbed semigroup, implies exponential stability of the semigroup generated by $A+B$.
\begin{thm}\label{main}
    Let $A$ be the generator of an analytic semigroup $\left(e^{tA}\right)_{t\geq0}$ of bounded linear operators on the Banach space $X$. Consider a linear operator $B:D(B)\to X$ and assume that the following hypotheses are true
    \begin{itemize}
        
        \item[\emph{\textbf{H1)}}]$B$ is $\theta$-subordinate to $A$ in the sense of Definition \ref{theta-sub}.

        \item[\emph{\textbf{H2)}}]$\forall\,t>0$, $Be^{tA}$ is compact.  

        \item[\hypertarget{H3}{\emph{\textbf{H3)}}}]$\left(e^{tA}\right)_{t\geq0}$ is uniformly exponentially stable, that is, $$\exists\, M_A \geq1,\,\,\omega_A>0\quad\text{such that}\quad \|e^{tA}\|_{\mathcal{L}}\leq M_Ae^{-\omega_A t},\,\,\forall\, t>0.$$\item[\hypertarget{H4}{\emph{\textbf{H4)}}}] \label{H4} $\left(e^{t(A+B)}\right)_{t\geq0}$ is strongly stable, that is,

        $$\lim_{t\to +\infty}\|e^{t(A+B)}x\|=0, \quad\forall\,x\in X.$$

    \end{itemize}
    Then, $\left(e^{t(A+B)}\right)_{t\geq0}$ is uniformly exponentially stable, that is,
    $$\exists\, \widehat{M}\geq1,\,\,\widehat{\omega}>0\quad\text{such that}\quad \|e^{t(A+B)}\|_{\mathcal{L}}\leq \widehat{M}e^{-\widehat{\omega} t},\,\,\forall\, t>0.$$
\end{thm}

\begin{proof}
By hypothesis \hypref{H3}{H3)}, 
we can choose $t_0>0$ sufficiently large such that $M_Ae^{-\omega_A t_0}<1$. Then
$$
\|e^{t_0A}\|_{\mathrm{ess}}
\leq
\|e^{t_0A}\|_{\mathcal L}
\leq
M_Ae^{-\omega_A t_0}
<1.
$$
Hence $\left(e^{tA}\right)_{t\geq0}$ is quasi-compact; see
\cite[Proposition V.3.5]{EnNa}. Therefore, by Lemma
\ref{quasi-compact-preservatn}, 
the semigroup $\left(e^{t(A+B)}\right)_{t\geq0}$ is quasi-compact.

We now prove that the quasi-compactness of
$\left(e^{t(A+B)}\right)_{t\geq0}$, together with its strong stability,
implies uniform exponential stability. To this aim, first consider
$$
\mathbb{S}:=
\{\mu\in\sigma(A+B):\Re\mu\geq 0\}.
$$
Since $\left(e^{t(A+B)}\right)_{t\geq0}$ is quasi-compact, the spectral
decomposition theorem for quasi-compact semigroups
(Theorem \ref{thm-quasicompact-spectral-decomp}) implies that the set $\mathbb{S}$ either is empty or consists only of finitely many isolated spectral values $\mu_1,\ldots,\mu_m$. In the latter case, each of these spectral values is a pole of the resolvent $R(\lambda,A+B)$
with finite algebraic multiplicity. If we denote by $U_1,\ldots,U_m$ the spectral projections corresponding to each pole, then the semigroup admits the decomposition~\eqref{spectdecomp},where the remaining part
$R(t)$ decays exponentially in the operator norm; that is, there exist
constants $\widehat M\geq1$ and $\widehat\omega>0$ such that
\begin{equation}\label{expo-remain}
    \|R(t)\|_{\mathcal{L}}\leq \widehat M e^{-\widehat\omega t},
    \qquad t\geq 0.
\end{equation}

We claim that $\mathbb{S}=\emptyset$. Suppose, on the contrary, that
there exists $\mu_i\in\mathbb{S}$. Since $\mu_i$ is a pole of
$R(\lambda,A+B)$ of finite algebraic multiplicity, the corresponding
spectral projection $U_i$ is non-zero, and the spectral subspace $U_iX$ is
finite-dimensional and non-trivial. Moreover, $U_iX$ is invariant under
$A+B$, and the restricted operator
$
(A+B)|_{U_iX}
$
has spectrum
$$
\sigma\left((A+B)|_{U_iX}\right)=\{\mu_i\}.
$$
Since $U_iX$ is finite-dimensional, $\mu_i$ is an eigenvalue of
$(A+B)|_{U_iX}$. Hence there exists $0\neq x_i\in U_iX$ such that $(A+B)x_i=\mu_i x_i$. Therefore,
$
e^{t(A+B)}x_i=e^{\mu_i t}x_i$ for all $t\geq 0$. In particular,
$$\|e^{t(A+B)}x_i\|
=e^{\Re\mu_i t}\|x_i\|,\qquad t\geq0.$$
Since $\Re\mu_i\geq0$, the right-hand side does not
converge to $0$ as $t\to\infty$. This contradicts hypothesis
\hypref{H4}{H4)}. 
Hence,
$
\mathbb{S}=\emptyset.
$

Now, applying Theorem \ref{thm-quasicompact-spectral-decomp} again, there are no spectral
terms corresponding to spectral values with non-negative real part. Hence
the decomposition reduces to its exponentially decaying remainder. Therefore,
by \eqref{expo-remain}, there exist constants $\widehat M\geq1$ and
$\widehat\omega>0$ such that
$$
\|e^{t(A+B)}\|_{\mathcal L}
\leq
\widehat M e^{-\widehat\omega t},
\qquad t\geq0.
$$
Thus $\left(e^{t(A+B)}\right)_{t\geq0}$ is uniformly exponentially stable. 
\end{proof}
Using Corollary \ref{crandall-pazy-Quasi-Compact}, we can similarly state the following corollary.

\begin{Cor}\label{main-cor}
     Let $A$ be the generator of a uniformly exponentially stable semigroup $\left(e^{tA}\right)_{t\geq0}$ of Crandall-Pazy class  with parameter $\alpha\in(0,1]$ on a Banach space $X$. Consider a linear operator $B:D(B)\to X$ such that $B$ is $\theta$-subordinate to $A$ in the sense of Definition \ref{theta-sub}, with $0<\theta<\alpha\leq1$, and that $Be^{tA}$ is compact for all $t>0$. Then
$\left(e^{t(A+B)}\right)_{t\geq0}$ is uniformly exponentially stable provided it is strongly stable.
\end{Cor}

From the proofs of Corollary \ref{Cor2} and Theorem \ref{main}, one can establish the preservation of uniform exponential stabilization of immediately differentiable semigroups under bounded perturbations in the case of general Banach spaces.

\begin{Cor}
Let $A$ be the generator of an immediately differentiable uniformly exponentially stable
$C_0$-semigroup $\left(e^{tA}\right)_{t\geq0}$ on a Banach space $X$.
Let $B$ be a bounded linear operator such that $Be^{tA}$ is compact
for all $t>0$. Then $\left(e^{t(A+B)}\right)_{t\geq0}$ is uniformly exponentially stable,
provided it is strongly stable.
\end{Cor}

\begin{rk}
    For immediately differentiable semigroups, condition~\eqref{eqCP} is a sufficient condition for the semigroup to be a Gevrey semigroup of class $\delta$, for any~$\delta>1/\alpha$; see \cite[ Corollary 5.1.7]{Tay}. Hence, the above results (Lemma \ref{quasi-compact-preservatn} and Theorem \ref{main}) can be generalized to immediately differentiable semigroups of Gevrey regularity. For completeness, we recall the definition of a Gevrey semigroup below; see \cite[\S 5.1]{Tay}.
\end{rk}

\begin{Def}
A $C_0$-semigroup $\left(e^{tA}\right)_{t\ge 0}$ on a Banach space $X$ is of Gevrey
class $\delta > 0$ for $t > t_0$ iff $\left(e^{tA}\right)_{t\ge 0}$ is differentiable for $t > t_0$ and for every compact $K \hookrightarrow (t_0,+\infty)$ and any $\gamma > 0$
there exists a constant $C = C(\gamma,K)$ such that
$$
\|T^{(n)}(t)\|_{\mathcal{L}}\le C\gamma^n (n!)^\delta\quad \forall\, t\in K\, ,\ n = 0,1,2,\dots
$$
\end{Def}

\section{Application to a diffusion operator on a non-reflexive Banach space}
\label{app}
In this section, we apply Theorem \ref{main} to a one-dimensional perturbed diffusion
operator on the non-reflexive Banach space $X:=C([0,1]),$ equipped with the usual norm~$\|\cdot\|_\infty.$

Let $\mu>0$ and
consider the operator $A_\mu f:=f_{xx}-\mu f$ with domain
$$
D(A_\mu):=
\left\{
f\in C^2([0,1]): f'(0)=f'(1)=0
\right\}.
$$
Equivalently, $A_\mu=A-\mu I$, where $A$ denotes the one-dimensional
Neumann Laplacian $Af=f_{xx}$ on $X$. It is known that $A$ generates a contraction $C_0$-semigroup $\left(e^{tA}\right)_{t\geq0}$ (see \cite[Example II.2.12]{EnNa}) which, for any $f\in C([0,1])$ and $x\in[0,1]$, is represented as
\begin{equation}\label{semigRep}
\!\!\! e^{tA}f(x):=
\begin{cases}
\displaystyle
\frac{1}{\sqrt{4\pi t}}\int_0^1
\sum_{m\in\mathbb Z}
\left\{
e^{-\frac{(x+r+2m)^2}{4t}}
+
e^{-\frac{(x-r+2m)^2}{4t}}
\right\}
f(r)\,\mathrm{d}r,
& t>0,\\
f(x),
& t=0.
\end{cases}
\end{equation}

\begin{Prop}\label{prop-exponential-analyticSemi}
$A_\mu$ generates a uniformly exponentially stable analytic semigroup $\left(e^{tA_\mu}\right)_{t\geq 0}$ of bounded linear operators on $X$.
\end{Prop}
\begin{proof}
Note that, since $\left(e^{tA}\right)_{t\geq 0}$ is a contraction semigroup, we have
$$
\|e^{tA_\mu}\|_{\mathcal{L}}\leq e^{-\mu t}\|e^{tA}\|_{\mathcal{L}}\leq e^{-\mu t}.
$$
Thus, $\left(e^{tA_\mu}\right)_{t\geq 0}$ is uniformly exponentially stable.

Next, in order to show that $\left(e^{tA_\mu}\right)_{t\geq 0}$ is analytic, from \cite[Theorem II.4.6]{EnNa}, it is enough to show that $\mathrm{R}\left(e^{tA_\mu}\right)\subset D(A)$ and that, for all $t>0,$ we have 
    \begin{equation}\label{Toshowforanalyticity}
        \sup_{t>0}\|tA_\mu e^{tA_\mu}\|_{\mathcal L}< \infty.
    \end{equation} 
Clearly, the differentiation of the heat kernel in
\eqref{semigRep} shows that $\mathrm{R}(e^{tA})\subset D(A),
\,\, t>0.$
Since $D(A_\mu)=D(A)$ and
$e^{tA_\mu}=e^{-\mu t}e^{tA}$, we also have $\mathrm{R}(e^{tA_\mu})\subset D(A_\mu),\,\, t>0.$ Further, from \eqref{semigRep}, we have that, for all $t>0,$
    \begin{align*}
        |tA e^{tA}f(x)|\leq &\frac{\|f\|_{\infty}}{2}\sqrt{\frac{t}{\pi}}\frac{\mathrm{d}^2}{\mathrm{d}x^2}\int_0^1
\sum_{m\in\mathbb Z}
\left\{
e^{-\frac{(x+r+2m)^2}{4t}}
+
e^{-\frac{(x-r+2m)^2}{4t}}
\right\}
f(r)\,\mathrm{d}r\nonumber\\
\leq &\frac{\|f\|_{\infty}}{2}\sqrt{\frac{t}{\pi}}\Bigg[
    \int_0^1
    \sum_{m\in\mathbb Z}
    \left|
    \frac{(x+r+2m)^2}{4t^2}
    -
    \frac{1}{2t}
    \right|
    e^{-\frac{(x+r+2m)^2}{4t}}
    \,\mathrm{d}r \nonumber\\
    &\quad+
    \int_0^1
    \sum_{m\in\mathbb Z}
    \left|
    \frac{(x-r+2m)^2}{4t^2}
    -
    \frac{1}{2t}
    \right|
    e^{-\frac{(x-r+2m)^2}{4t}}
    \,\mathrm{d}r\Bigg] . 
    \end{align*}
    For fixed $x\in[0,1]$, by the changes of variables
$y=x+r+2m$ and $y=x-r+2m$, respectively, we get
\begin{align}\label{weGetUseful}
    \left|tA e^{tA}f(x)\right|
    &\leq
    \|f\|_{\infty}\sqrt{\frac{t}{\pi}}
    \int_{\mathbb R}
    \left|
    \frac{y^2}{4t^2}
    -
    \frac{1}{2t}
    \right|
    e^{-\frac{y^2}{4t}}
    \,\mathrm{d}y ,
\end{align}
where we have used the periodicity of the sum inside the integrand, that is,

\begin{align*}
&\int_0^1 \sum_{m\in\mathbb Z}
\left|
\frac{(x\pm r+2m)^2}{4t^2}
-
\frac{1}{2t}
\right|
e^{-\frac{(x\pm r+2m)^2}{4t}}
\,\mathrm{d}r\\
&=
\begin{cases}
\displaystyle
\int_x^{x+1} \sum_{m\in\mathbb Z}
\left|
\frac{(y+2m)^2}{4t^2}
-
\frac{1}{2t}
\right|
e^{-\frac{(y+2m)^2}{4t}}
\,\mathrm{d}y,
& \text{for }y= x+r+2m,\\
\displaystyle
\int_{x-1}^{x} \sum_{m\in\mathbb Z}
\left|
\frac{(y+2m)^2}{4t^2}
-
\frac{1}{2t}
\right|
e^{-\frac{(y+2m)^2}{4t}}
\,\mathrm{d}y,
& \text{for }y=x-r+2m,
\end{cases}
\end{align*}
which gives
\begin{multline*}
\int_0^1 \sum_{m\in\mathbb Z}
\left|
\frac{(x\pm r+2m)^2}{4t^2}
-
\frac{1}{2t}
\right|
e^{-\frac{(x\pm r+2m)^2}{4t}}
\,\mathrm{d}r\\
\leq
\int_0^2 \sum_{m\in\mathbb Z}
\left|
\frac{(y+2m)^2}{4t^2}
-
\frac{1}{2t}
\right|
e^{-\frac{(y+2m)^2}{4t}}
\,\mathrm{d}y\\
=
\sum_{m\in\mathbb Z}
\int_{2m}^{2m+2}
\left|
\frac{y^2}{4t^2}
-
\frac{1}{2t}
\right|
e^{-\frac{y^2}{4t}}
\,\mathrm{d}y=
\int_{\mathbb R}
\left|
\frac{y^2}{4t^2}
-
\frac{1}{2t}
\right|
e^{-\frac{y^2}{4t}}
\,\mathrm{d}y.
\end{multline*}
Using the change of variable $z=y/(2\sqrt{t})$ in  \eqref{weGetUseful}, we obtain
\begin{align*}\label{useitlater}
     \left|tA e^{tA}f(x)\right|
    &\leq
    \frac{2\|f\|_{\infty}}{\sqrt{\pi}}
    \int_{\mathbb R}
    \left|
    z^2-\frac12
    \right|
    e^{-z^2}\,\mathrm{d}z <\infty.
\end{align*}
By taking the supremum over $x\in[0,1]$, 
we obtain that
\begin{equation*}
    \|tAe^{tA}f\|_{\infty}
    \leq
    \frac{2\|f\|_{\infty}}{\sqrt{\pi}}
    \int_{\mathbb R}
    \left|z^2-\frac12\right|
    e^{-z^2}\,\mathrm{d}z,
    \qquad t>0.
\end{equation*}
Hence,
\begin{equation*}\label{uniA}
    \sup_{t>0}\|tAe^{tA}\|_{\mathcal L}
    \leq
    \frac{2}{\sqrt{\pi}}
    \int_{\mathbb R}
    \left|z^2-\frac12\right|
    e^{-z^2}\,\mathrm{d}z
    <\infty.
\end{equation*} 
In combination with 
the contractivity of
$\left(e^{tA}\right){t\geq0}$, we get that
\begin{multline*} \sup_{t>0}\|tA_\mu e^{tA_\mu}\|_{\mathcal L}=\sup_{t>0}\left[\|t(A-\mu I)e^{t(A-\mu I)}\|_{\mathcal L}\right]\\ \leq\sup_{t>0}\left[e^{-\mu t}\|tAe^{tA}\|_{\mathcal L}+\mu te^{-\mu t}\|e^{tA}\|_{\mathcal L}\right]<\infty, \end{multline*}
which proves~\eqref{Toshowforanalyticity}, and thus $\left(e^{tA_\mu}\right)_{t\geq0}$ is an analytic semigroup.
\end{proof}

Since $f\in C^2([0,1])\subset W^{2,\infty}(0,1)$, by  virtue of the Gagliardo-Nirenberg interpolation inequality for bounded domains
\cite[Theorem 1.2]{LiZhang2022GN}, we obtain the following lemma.
\begin{Lem}\label{interpoineq}
For $f\in C^2([0,1])$, there exists a constant
$C>0$ such that
$$
\|f_x\|_\infty
\leq
C\left(
\|f\|_\infty^{1/2}\|f_{xx}\|_\infty^{1/2}
+
\|f\|_\infty
\right). 
$$
\end{Lem}
We now define the perturbation operator $B: D(B)\subset X\to X$ by
$$
Bf:=p(x)f_x(x),
\qquad
D(B):=C^1([0,1]),
$$
where $p\in C^1([0,1])$ is a real-valued function satisfying
\begin{equation}\label{hyp-p}
\left\{
\begin{array}{l}
p(0)=p(1)=0,\\
p_x(x)\geq -2\mu,
\qquad x\in[0,1].
\end{array}
\right.
\end{equation}
\begin{Prop}\label{prop2}
    $B$ is $1/2$-subordinate to $A_\mu$ in the sense of Definition \ref{theta-sub}.
\end{Prop}

\begin{proof}
    Clearly $D(A_\mu)\subset D(B)$.  For every $f\in D(A_\mu)$, from Lemma \ref{interpoineq} we obtain
\begin{align}\label{56}
\|Bf\|_\infty
&=
\|p f_x\|_\infty \leq
\|p\|_\infty\|f_x\|_\infty\nonumber\\
&\leq
C\|p\|_\infty
\left(
\|f\|_\infty^{1/2}\,\|f_{xx}\|_\infty^{1/2}
+
\|f\|_\infty
\right).
\end{align}
Note that the elliptic maximum principle ensures that
\begin{equation}\label{ellpmaxinq}
    \|f\|_\infty
\leq
\frac1\mu\|A_\mu f\|_\infty,
\qquad f\in D(A_\mu).
\end{equation} Indeed, let $x_0\in[0,1]$ be such that $f(x_0)=\max_{x\in[0,1]} f(x).$ Since $f'(0)=f'(1)=0$, by the maximum principle we have that $f_{xx}(x_0)\leq0.$ Therefore, $(A_\mu f)(x_0)
= f_{xx}(x_0)-\mu f(x_0)
\leq
-\mu f(x_0)$, and so $\mu f(x_0)
\leq
-\left(A_\mu f\right)(x_0)
\leq
\|A_\mu f\|_\infty.$ Applying a similar argument to $-f$ gives \eqref{ellpmaxinq}. 
Therefore,
$$
\|f_{xx}\|_\infty
\leq
\|A_\mu f\|_\infty+\mu\|f\|_\infty
\leq
2\|A_\mu f\|_\infty,
$$
which, together with \eqref{56} implies that
$$
\|Bf\|_\infty
\leq
C\|p\|_\infty\left(\sqrt{2}+\frac1{\sqrt{\mu}}\right)
\|A_\mu f\|_\infty^{1/2}
\|f\|_\infty^{1/2},
\qquad f\in D(A_\mu).$$
Thus $B$ is $1/2$-subordinate to $A_\mu$ in the sense of Definition \ref{theta-sub}.
\end{proof}
\begin{Prop}\label{prop3}
    $Be^{tA_\mu}$ is compact for all $t>0.$
\end{Prop}
\begin{proof}
   Fix $t>0$. We will prove that $Be^{tA_\mu}$ maps bounded subsets of $C([0,1])$
into relatively compact subsets of $X$. Let $\mathscr B:=\{f\in C([0,1])\,\,|\,\,\|f\|_\infty\leq 1\}$ and the family $\mathscr F:=\{Be^{tA_\mu}f \,\,|\,\,f\in \mathscr B\}$. 

Note that $$
Be^{tA_\mu}f(x)
=
p(x)e^{-\mu t}\int_0^1\Gamma(t,x,\xi)f(\xi)\,d\xi .
$$ where
$$ \Gamma(t,x,\xi) = -\frac{ \sum_{m\in\mathbb Z} \left\{ (x+\xi+2m)e^{-\frac{(x+\xi+2m)^2}{4t}} + (x-\xi+2m)e^{-\frac{(x-\xi+2m)^2}{4t}} \right\}}{2t\sqrt{4\pi t}}, $$ due to \eqref{semigRep}. Since $p(x)e^{-\mu t}{\Gamma(t,x,\xi)}=:\Gamma_t(x,\xi)\in C([0,1]\times[0,1])$, for every $f\in\mathscr B$ and every $x\in[0,1]$,
we have
$$|Be^{tA_\mu}f(x)|
\leq\int_0^1 |\Gamma_t(x,\xi)|\,|f(\xi)|\,d\xi
\leq
\|\Gamma_t\|_\infty .$$ Thus, the family $\mathscr F$ is uniformly bounded in $X$.

Furthermore, for any $x,y\in[0,1]$ and $f\in\mathscr B$,
we obtain
$$
\begin{aligned}
|Be^{tA_\mu}f(x)-Be^{tA_\mu}f(y)|
&\leq
\int_0^1 |\Gamma_t(x,\xi)-\Gamma_t(y,\xi)|\,|f(\xi)|\,d\xi  \\
&\leq
\sup_{\xi\in[0,1]}|\Gamma_t(x,\xi)-\Gamma_t(y,\xi)|.
\end{aligned}
$$
Since $\Gamma_t$ is continuous on the compact set $[0,1]\times[0,1]$, it is uniformly
continuous. Therefore,
$$
\sup_{\xi\in[0,1]}|\Gamma_t(x,\xi)-\Gamma_t(y,\xi)|
\longrightarrow 0
\qquad \text{as }\ |x-y|\to0.
$$
Therefore, $\mathscr F$ is also equicontinuous. Thus, thanks to the Arzelà-Ascoli theorem~\cite[Theorem 6.12]{Robinson2020FA}, $\mathscr F$ is relatively compact in
$C([0,1])$. Therefore, $Be^{tA_\mu}$ is a compact operator on $X$ for every $t>0$. 
\end{proof}

We now prove the uniform exponential stability of the following system:
\begin{equation}\label{perturbedsystem}
\left\{
\begin{array}{ll}
u_t(t,x)=u_{xx}(t,x)+p(x)u_x(t,x)-\mu u(t,x),
& x\in(0,1),\ t>0,\\
u_x(t,0)=u_x(t,1)=0,
& t>0,\\
u(0,x)=u_0(x),
& x\in[0,1],
\end{array}
\right.
\end{equation}
where $\mu>0$ and $p(x)$ satisfies assumption \eqref{hyp-p}. System~\eqref{perturbedsystem} can be recast as the abstract Cauchy problem \begin{equation}\label{abstractperturbedsystem} \left\{ \begin{array}{ll} \dot{u}(t)=(A_\mu+B)u(t), & t>0,\\[1mm] u(0)=u_0\in X, \end{array} \right. \end{equation}
\begin{rk}
    The system \eqref{perturbedsystem} may be interpreted as 1D
diffusion-drift-reaction model on a bounded medium with homogeneous Neumann boundary conditions, which describe insulated  endpoints. The term $u_{xx}$ represents
diffusion, while the zeroth-order term $-\mu u$,
with $\mu>0$, models the reaction term. The term $p(x)u_x$ describes transport due to the drift field $p$.  
\end{rk}

\begin{thm}
The operator $A_\mu+B$ generates a uniformly exponentially stable analytic
semigroup $\left(e^{t(A_\mu+B)}\right)_{t\geq0}$ on $X$.
\end{thm}

\begin{proof}
By Propositions \ref{prop-exponential-analyticSemi}, \ref{prop2} and
\ref{prop3}, hypotheses 
\hypref{H1}{H1)},
\hypref{H2}{H2)} and
\hypref{H3}{H3)} are satisfied. Hence, by
Theorem \ref{main}, the semigroup $\left(e^{t(A_\mu+B)}\right)_{t\geq0}$ is
uniformly exponentially stable, provided it is strongly stable.

It remains to verify the strong stability of $\left(e^{t(A_\mu+B)}\right)_{t\geq0}$. Owing to LaSalle's invariance principle, it is sufficient to show that the elliptic boundary value problem
\begin{equation}\label{ellipbound}
    \left\{ \begin{array}{ll} u_{xx}+p(x)u_x-\mu u=0, & x\in(0,1),\\ u_x(0)=u_x(1)=0, \end{array} \right.
\end{equation} admits only the trivial solution $u\equiv0$. To this aim, let $u\in C^2([0,1])$ be a solution to~\eqref{ellipbound}. Multiplying the equation by $u$ and integrating over $(0,1)$, we obtain $$ \int_0^1 u u_{xx}\,\mathrm{d}x + \int_0^1 p(x)u u_x\,\mathrm{d}x - \mu\int_0^1 u^2\,\mathrm{d}x =0. $$ Using integration by parts and the boundary condition in~\eqref{ellipbound}, we get $$ \int_0^1 u u_{xx}\,\mathrm{d}x = -\int_0^1 |u_x|^2\,\mathrm{d}x. $$ Moreover, since $p(0)=p(1)=0$, we also have $$ \int_0^1 p(x)u u_x\,\mathrm{d}x = \frac12\int_0^1 p(x)(u^2)_x\,\mathrm{d}x = -\frac12\int_0^1 p_x(x)u^2\,\mathrm{d}x. $$ Therefore, \begin{equation}
    \label{cdf}
    \int_0^1 |u_x|^2\,\mathrm{d}x + \int_0^1\left(\mu+\frac12\, p_x(x)\right)u^2\,\mathrm{d}x =0.
\end{equation}
Assumption~\eqref{hyp-p} ensures that $p_x(x)\geq-2\mu$ for all $x\in[0,1]$, and therefore both terms in \eqref{cdf} are nonnegative. We first obtain
$$\int_0^1|u_x|^2\,\mathrm{d}x=0,$$which implies that $u$ is constant, say $u(x)=c$ for all $x\in[0,1]$. Upon substituting this into the second nonnegative term in \eqref{cdf}, we get $$c^2\int_0^1\left(\mu+\frac12\, p_x(x)\right)\,\mathrm{d}x =0,$$ which implies $c^2(2\mu+(p(1)-p(0)))=0.$
As $\mu\neq0$ and, by assumption \eqref{hyp-p}, $p(0)=p(1)=0$, it follows that $c=0$. Hence, $u\equiv0$ on $[0,1]$. Therefore, applying Theorem \ref{main}, we conclude that system \eqref{perturbedsystem} is uniformly exponentially stable.
\end{proof}
\begin{rk}
Note that assumption \eqref{hyp-p} rules out the case $p_x(x)=-2\mu$ on the entire interval $[0,1]$. Indeed, if this were the case, then $0=p(1)-p(0)=\int_0^1p_x(x)\,\mathrm{d}x=-2\mu,$ which contradicts the fact that $\mu>0$. 
This differs from the assumption on the perturbation of uniformly parabolic problems on reflexive Banach spaces considered in~\cite[Assumption 3.1.4, \S 3.1]{SLRRS-PC-RG}, where an additional strict inequality involving the coefficient of the first-order term was imposed on a nonempty open subset of $[0,1]$, and unique
continuation \cite{Hormander63} was then invoked to conclude that the solution vanished throughout the domain. In particular, the present argument avoids the use of a unique continuation principle. 
\end{rk}
\section{Lack of uniform stabilization for $A$-bounded perturbations of generators}\label{sec4}

It is important to note that the proof of Theorem \ref{main} cannot be directly extended to $A$-bounded perturbations in the sense of Definition \ref{Abound}.  The obstruction appears in Step \hyperlink{stp5}{{\textbf{5}}} of the proof of Lemma \ref{quasi-compact-preservatn}. In the analytic
case, the estimate~\eqref{Analytic-A} combined with $\theta$-subordination gives~\eqref{B-T-estimate-2}, that is,
\[
\|e^{(t-s)(A+B)}Be^{sA}x\|
   \leq C(t)s^{-\theta}\|x\|\,,\,\, 0<s\leq t,\,\, x\in X,
\]
and the singularity $s^{-\theta}$ is integrable near $s=0$ because
$\theta<1$. 
However, for a merely A-bounded perturbation in the sense of Definition~\ref{Abound}, the estimate $$\|Be^{sA}x\|
\leq
a\|Ae^{sA}x\|+b\|e^{sA}x\|$$
together with \eqref{Analytic-A} gives
$$\|Be^{sA}x\|
\leq
C(t)\left(1+\frac1s\right)\|x\|.$$ Consequently, the analogue of estimate~\eqref{B-T-estimate-2} would contain the
non-integrable singularity $s^{-1}$. Therefore, the convergence cannot be obtained by the same argument. 
In fact, such stabilization cannot hold for $A$-bounded perturbations in general. We provide below three distinct counterexamples.
   
\begin{CE}\label{cont1}
There exist a Hilbert space $H$, a generator $A$ of a uniformly exponentially stable analytic semigroup $\left(e^{tA}\right)_{t\geq0}$ on $H$, and $A$-bounded perturbations $B:D(A)\to H$ that are not $\theta$-subordinate to $A$ for any $\theta\in[0,1)$, such that $Be^{tA}$ is compact for every $t>0$, and the semigroup generated by $A+B$ is strongly stable but not uniformly exponentially stable.

\end{CE}

\begin{proof}
The proof is by construction. Let $H=\ell^2(\mathbb N; \C)$ and define the operator $A:D(A)\subset H\to H$ by
$Ax=(-nx_n)_{n\in\mathbb N},$
with domain
$$D(A)=\left\{
x=(x_n)_{n\in\mathbb N}\in \ell^2(\mathbb N) \,\Bigg|\,
\sum_{n=1}^{\infty} n^2|x_n|^2<\infty
\right\}.$$
Note that $A$ is the multiplication operator $M_q$ on $\ell^2(\mathbb N)$ generated by the sequence $q= (-n)_{n\in\mathbb N}$, such that $M_qx= (-nx_n)_{n\in\mathbb N}.$ It is well known that the operator $A=M_q$ generates the strongly continuous multiplication
semigroup $$e^{tA}x
=
\left(e^{-nt}x_n\right)_{n\in\mathbb N},
\qquad t\geq0,$$
(see \cite[\S I.4]{EnNa}). Moreover, since $\sigma(A)
=
\{-n:n\in\mathbb N\}
\subset(-\infty,0)$, for every $\delta\in(0,\pi/2)$ we have
$\Sigma_{\delta+\frac{\pi}{2}}\cap\sigma(A)=\emptyset.$
Therefore, $\Sigma_{\delta+\frac{\pi}{2}}
\subset
\mathbb C\setminus\sigma(A)
=
\rho(A),$ where
$$\Sigma_{\delta+\frac{\pi}{2}} :=\left\{\lambda\in\C\,\, \Big |\,\, |\arg \lambda|\leq \frac{\pi}{2}+\delta\right\}-\{0\}.$$
It follows from Theorem in \cite[\S II.4.32(i)]{EnNa} that
$\left(e^{tA}\right)_{t\geq0}$ is a bounded analytic semigroup of angle
$\delta$ for every $\delta\in(0,\pi/2)$. 
Furthermore, for every $x\in H$, we have
$$\|e^{tA}x\|^2=\sum_{n=1}^{\infty} e^{-2nt}|x_n|^2\leq e^{-2t}\sum_{n=1}^{\infty} |x_n|^2= e^{-2t}\|x\|^2\qquad t\geq0,$$
and therefore $\left(e^{tA}\right)_{t\geq0}$ is uniformly exponentially
stable.

Now we construct a class of perturbations $B:D(A)\to H$ of the form 
\[
Bx=\left(\left(n-\frac{1}{n^{\beta}}+i \,n\right)x_n\right)_{n\in\mathbb{N}},\quad \beta>0,\,\ i \,^2=-1,
\]
such that $B$ is $A$-bounded but not $\theta$-subordinate for any $\theta\in[0,1)$. Indeed, for $x\in D(A)$ and $\beta>0$, we have
\begin{multline*}
    \|Bx\|^2=
\sum_{n=1}^\infty
\left|n-\frac{1}{n^{\beta}}+i \,n\right|^2 |x_n|^2 \\=
\sum_{n=1}^\infty
\left(\left(1-\frac{1}{n^{\beta+1}}\right)^2+1\right)n^2|x_n|^2 \leq 2\sum_{n=1}^\infty n^2|x_n|^2
=
2\|Ax\|^2.
\end{multline*}
Hence $B$ is $A$-bounded.

Now consider, for each $n\in\mathbb N$,  
\begin{equation}\label{counterdelta}
    \textbf{e}_n=(0,0,\ldots,0,1,0,\ldots),
\end{equation}
where the entry $1$ appears in the $n$-th position. Then $\textbf{e}_n\in D(A)$, $\|\textbf{e}_n\|=1$, and
$\|A\textbf{e}_n\|=n.$ Moreover, we have
$
\|B\textbf{e}_n\|
= \left|n-\frac1{n^\beta}+i \,n\right|.
$
Therefore, if $B$ were $\theta$-subordinate to $A$ for some $c>0$ and
$\theta\in[0,1)$, then we would have
\begin{equation*}\label{thetasubreq}
    \left|n-\frac1{n^\beta}+i \,n\right|
    =
    \left(\left(n-\frac1{n^\beta}\right)^2+n^2\right)^{1/2}
\leq
c\|A\textbf{e}_n\|^\theta\|\textbf{e}_n\|^{1-\theta}
=
cn^\theta,
\end{equation*}
for every $n\in\mathbb N$. Dividing both sides by $n^\theta$, we obtain
\begin{equation}\label{c-bound}
    \frac{\left(\left(n-\frac1{n^\beta}\right)^2+n^2\right)^{1/2}}{n^\theta}
\leq c,
\qquad n\in\mathbb N.
\end{equation}
However,
$$
\frac{\left(\left(n-\frac1{n^\beta}\right)^2+n^2\right)^{1/2}}{n^\theta}
=
n^{1-\theta}
\left(\left(1-\frac1{n^{\beta+1}}\right)^2+1\right)^{1/2}
\geq n^{1-\theta}
\to\infty,
$$
as $n\to\infty,$ because $\theta<1$. This contradicts the boundedness of the left-hand side by the fixed constant $c$ in \eqref{c-bound}. Hence $B$ cannot be $\theta$-subordinate to $A$ for any $\theta\in[0,1)$.

We now want to prove that $Be^{tA}$ is compact for every $t>0$. For any $t>0$, and $x\in H$, the operator $Be^{tA}$ is given by
$$
Be^{tA}x
=
\left(\left(n-\frac1{n^\beta}+i \,n\right)e^{-nt}x_n\right)_{n\in\mathbb N}.
$$
Recall that an operator $Kx=(a_nx_n)_{n\in\mathbb N}$ from
$\ell^2(\mathbb N)$ into itself is compact if and only if
$\lim_{n\to\infty}|a_n|=0$~\cite[Example 4, \S 5.24]{NaylorSell1982}. Since
$$
\left|\left(n-\frac1{n^\beta}+i \,n\right)e^{-nt}\right|
\to0
\qquad\text{as }n\to\infty,
$$
the operator $Be^{tA}$ is compact on $\ell^2(\mathbb N)$.

Finally, the perturbed operator is given by
$$
(A+B)x
=
\left(\left(-\frac{1}{n^\beta}+i \,n\right)x_n\right)_{n\in\mathbb N},$$
with $D(A+B)=D(A).$ Thus $A+B$
generates the $C_0$-semigroup
$$e^{t(A+B)}x=\left(e^{\left(-\frac{1}{n^\beta}+i \,n\right)t}x_n\right)_{n\in\mathbb N},
\qquad t\geq0.$$
Observe that, for every $x\in H$, we have
\begin{equation}\label{4strong}
\|e^{t(A+B)}x\|^2=\sum_{n=1}^\infty e^{-\frac{2t}{n^\beta}}|x_n|^2,
\end{equation}
and for each fixed $n$, $e^{-2t/n^\beta}\to0$ as $t\to\infty$. Moreover, we have 
\begin{equation}\label{5strong}
    e^{-\frac{2t}{n^\beta}}|x_n|^2\leq |x_n|^2.
\end{equation}
Therefore, by the dominated convergence theorem, \eqref{5strong} and \eqref{4strong} imply the strong stability of $\left(e^{t(A+B)}\right)_{t\geq0}$, that is,
$$\|e^{t(A+B)}x\|\to0
\qquad\text{as }t\to\infty.$$

On the other hand, we show that the semigroup
$\left(e^{t(A+B)}\right)_{t\geq0}$ is not uniformly exponentially stable. We argue by contradiction. 
Suppose that there exist $M\geq1$ and $\omega>0$ such that
\[
\|e^{t(A+B)}x\|
\leq
Me^{-\omega t}\|x\|,
\qquad t\geq0,\ x\in\ell^2(\mathbb N).
\]
For each $n\in\mathbb N$, let $\textbf{e}_n$ be as in \eqref{counterdelta}. Then
$
e^{t(A+B)}\textbf{e}_n
=
e^{\left(-\frac1{n^\beta}+in\right)t}\textbf{e}_n.
$
Therefore, by the Cauchy--Schwarz inequality, we have
$$
e^{-t/n^\beta}
=
\left|
\left\langle e^{t(A+B)}\textbf{e}_n,\textbf{e}_n\right\rangle
\right|
\leq
\|e^{t(A+B)}\textbf{e}_n\|\|\textbf{e}_n\|
\leq
Me^{-\omega t}. $$
Taking $t=n^\beta$, we obtain
$
e^{-1}\leq Me^{-\omega n^\beta},
\, n\in\mathbb N$, where the right-hand side goes to zero as $n\to\infty$. This gives a contradiction, and hence
$\left(e^{t(A+B)}\right)_{t\geq0}$ is not uniformly exponentially stable. 
\end{proof}

\begin{CE}\label{cont2}
There exist a reflexive Banach space $X$ and a family of generators $A$ of uniformly exponentially stable $C_0$-semigroups on $X$, as well as $A$-bounded perturbations $B:D(A)\to X$, such that $B$ is not $\theta$-subordinate to $A$ for any $\theta\in[0,1)$, while the semigroup generated by $A+B$ is strongly stable but not uniformly exponentially stable.
\end{CE}
\begin{proof}
We construct a class of counterexamples on $X=L^p(\mathbb R_+)$, $1<p<\infty$. Let $\alpha>1$ be fixed. Define the operator $A:D(A)\subset L^p(\mathbb R_+)\to L^p(\mathbb R_+)$ by

\begin{equation}\label{defA}
    Af=\alpha f'+(1-\alpha)f,
\end{equation}
with domain $D(A):=W^{1,p}(\mathbb R_+)$, and the perturbation $B:D(A)\to X$ by
\begin{equation}\label{defB}
    Bf=(1-\alpha)f'+(\alpha-1)f.
\end{equation} 
Then $(A+B)f=f'$ with $D(A+B) = D(A)$, which generates the left-translation semigroup \cite[Proposition 1, \S II.2.10]{EnNa} given by
$ e^{s(A+B)}f(x)=f(x+s),
\, s\geq0. $
Moreover, this semigroup is strongly stable but not uniformly exponentially stable
\cite[Example 1.3(i), \S V.1.a]{EnNa}. Furthermore, note that the  $C_0$-semigroup generated by $A$ can be represented by

\begin{equation}\label{e^{sA}f}
    e^{sA}f(x)=e^{(1-\alpha)s}f(x+\alpha s),
\qquad s\geq0,\ x\in\mathbb R_+.
\end{equation} This semigroup is uniformly exponentially stable. Indeed, \begin{multline*}
\|e^{sA}f\|^p= e^{p(1-\alpha)s}
\int_0^\infty |f(x+\alpha s)|^p\,dx \\
= e^{p(1-\alpha)s} \int_{\alpha s}^\infty |f(y)|^p\,dy
\leq e^{p(1-\alpha)s}\|f\|^p\, .
\end{multline*}
Therefore,
$ \|e^{sA}f\|
\leq e^{-(\alpha-1)s}\|f\|,
\, s\geq0.
$ Since $\alpha>1$, it follows that $\left(e^{sA}\right)_{s\geq0}$ is uniformly
exponentially stable.

It remains to show that $B$ is $A$-bounded but not $\theta$-subordinate to $A$
for any~$\theta\in[0,1)$. To this aim, note that, from \eqref{defA} and \eqref{defB}, we have 
$$
Bf=\frac{1-\alpha}{\alpha}Af+\frac{\alpha-1}{\alpha}f.
$$
Consequently, for every $f\in D(A)$, we have
$$
\|Bf\|
\leq
\frac{\alpha-1}{\alpha}\|Af\|
+
\frac{\alpha-1}{\alpha}\|f\|.
$$
Thus $B$ is $A$-bounded.
Now we show that $B$ is not $\theta$-subordinate to $A$ for any~$\theta\in[0,1)$. Let $\phi\in C_c^\infty(0,\infty)$ be such that
$\phi'\neq0$, and define
\begin{equation}\label{def_n}
    f_n(x)=\frac{n^{1/p}\phi(nx)}{\|\phi\|},
\qquad n\in\mathbb {N}.
\end{equation}

Then $f_n\in D(A)$, and it is easy to check that $\|f_n\|
= 1, \, \|f_n'\|
=n\|\phi'\|/{\|\phi\|}.$ Therefore, by the reverse triangle inequality, we get $$\|Af_n\|= \|\alpha f_n'-(\alpha-1)f_n\|\geq
\alpha\|f_n'\|
- (\alpha-1)\|f_n\|=
\alpha n\frac{\|\phi'\|}{\|\phi\|} - \alpha + 1,
$$
which implies that $\|Af_n\|\to\infty$ as $n\to\infty.$
Moreover, using the expression of~$B$ in terms of $A$ and applying the reverse triangle inequality  again, we obtain
\begin{equation}\label{reverseBfineq}
    \|Bf_n\|\geq \frac{\alpha-1}{\alpha}\|Af_n\|-\frac{\alpha-1}{\alpha}\|f_n\|.
\end{equation}
Suppose by contradiction that $B$ is $\theta$-subordinate to $A$ for some
$\theta\in[0,1)$. Then there exists a constant $c>0$ such that
$$
\|Bf\|
\leq
c\|Af\|^\theta
\|f\|^{1-\theta},
\qquad f\in D(A).
$$ Applying this inequality to $f=f_n$ and using \eqref{reverseBfineq}, we get \begin{align*}
  \frac{\alpha-1}{\alpha}\|Af_n\|-\frac{\alpha-1}{\alpha}\|f_n\|
  \leq c\|Af_n\|^\theta
\|f_n\|^{1-\theta}.
\end{align*}
Since $\|f_n\|=1$, we have $$
\frac{\alpha-1}{\alpha}\|Af_n\|^{1-\theta}
- \frac{\alpha-1}{\alpha}\|Af_n\|^{-\theta}
\leq c.
$$ Letting $n\to\infty$, this leads to a contradiction, since $\|Af_n\|\to\infty$ as $n\to\infty$, and~$\theta<1$. Hence $B$ is not $\theta$-subordinate to $A$ for any $\theta\in[0,1)$.
\end{proof}
\begin{rk}
    Note that, in Counterexample \ref{cont2}, $Be^{tA}$ is not
compact for any $t>0$. In fact, $Be^{tA}$ is not even bounded on
$L^p(\mathbb R_+)$. Indeed, from~\eqref{defB} and~\eqref{e^{sA}f}, we have $$Be^{tA}f(x) = (\alpha-1)e^{(1-\alpha)t} \left(-f'(x+\alpha t)+f(x+\alpha t)\right).$$
Fix $t>0$. Choose $\varphi\in C_c^\infty(0,\infty)$ with $\varphi'\neq0$, and define $$g_n(x)=
\frac{n^{1/p}\varphi(n(x-\alpha t))}{\|\varphi\|},
\qquad n\in\mathbb N.$$
Then $\|g_n(\cdot\,+\alpha t)\|=\|g_n\|=1$ and $\|g'_n(\cdot\,+\alpha t)\|=\|g_n'\|=n\|\varphi'\|/\|\varphi\|.$ Therefore, we have
\begin{multline*}
    \|Be^{tA}g_n\|=
(\alpha-1)e^{(1-\alpha)t}
\left\|-g_n'(\cdot+\alpha t)+g_n(\cdot+\alpha t)\right\| \\
\geq
(\alpha-1)e^{(1-\alpha)t}
\left(\|g_n'\|-\|g_n\|\right)\\
=(\alpha-1)e^{(1-\alpha)t}
\left(
n{\|\varphi'\|}/{\|\varphi\|}-1
\right)
\to\infty
\quad\text{as}\quad n\to\infty .
\end{multline*} Since $\|g_n\|=1$ for all $n\in\mathbb{N}$, it follows that $Be^{tA}$ is not bounded on
$L^p(\mathbb R_+)$. Hence $Be^{tA}$ cannot be compact for any $t>0$.
\end{rk}
\begin{rk} It is worth noting that the $C_0$-semigroup \eqref{e^{sA}f} generated by $A$ in Counterexample~\ref{cont2} is not analytic. Indeed, analyticity would require that $e^{tA}X\subset D(A)$ for every $t>0$ (see \cite[Theorem II.4.6]{EnNa}). Fix $t>0$ and choose \[ f_t:=\mathbf{1}_{[\alpha t+1,\,\alpha t+2]} \in L^p(\mathbb R_+), \] where $\mathbf{1}_{\mathcal{U}}$ denotes the indicator function of a set $\mathcal{U}$, defined by \[ \mathbf{1}_{\mathcal{U}}(x) := \begin{cases} 1, & x\in \mathcal{U},\\ 0, & x\notin \mathcal{U}. \end{cases} \] Then \[ e^{tA}f_t(x) = e^{(1-\alpha)t} \mathbf{1}_{[\alpha t+1,\alpha t+2]}(x+\alpha t) = e^{(1-\alpha)t}\mathbf{1}_{[1,2]}(x). \] Since $\mathbf{1}_{[1,2]} \notin W^{1,p}(\mathbb R_+)$, we obtain $e^{tA}f_t\notin D(A),$ and therefore $\left(e^{tA}\right)_{t\geq0}$ is not analytic.
\end{rk}
Finally, we show that uniform exponential stabilization  fails for $\theta$-subordinate perturbations when hypothesis \hypref{H2}{H2)} does not hold. This shows that compactness along the trajectories is an essential assumption in our stabilization result.
\begin{CE}\label{cont3}
There exist a Hilbert space $H$, a generator $A$ of a uniformly exponentially
stable analytic semigroup $\left(e^{tA}\right)_{t\geq0}$ on $H$, and a
$\theta$-sub\-or\-di\-nate perturbation $B:D(A)\to H$ such that $Be^{tA}$ is not compact for any ~$t>0$, and the semigroup generated by $A+B$ is strongly stable but not uniformly exponentially stable.
\end{CE}
\begin{proof}
The proof is by construction. Let $H=\ell^2(\mathbb N)$ and define
$A:H\to H$ by
$$
Ax=-x,\qquad x\in H.
$$
Then $D(A)=H$, and $A$ generates the analytic semigroup
$$
e^{tA}x=e^{-t}x,\qquad t\geq0.
$$
Moreover,
$$
\|e^{tA}x\|=e^{-t}\|x\|,
\qquad t\geq0,
$$
and therefore $\left(e^{tA}\right)_{t\geq0}$ is uniformly exponentially
stable. Now define $B:H\to H$ by
$$
Bx=\left(\left(1-\frac1n\right)x_n\right)_{n\in\mathbb N}.
$$
Then $B$ is bounded on $H$, and hence $B$ is $0$-subordinate to $A$.

Next, for every $t>0$, we have
$$
Be^{tA}x
=
Be^{-t}x
=
\left(e^{-t}\left(1-\frac1n\right)x_n\right)_{n\in\mathbb N}.
$$
Using the fact that a diagonal operator
$Kx=(a_nx_n)_{n\in\mathbb N}$ on $\ell^2(\mathbb N)$ is compact if and only if
$a_n\to0$ \cite[Example 4, \S 5.24]{NaylorSell1982}, we see that $Be^{tA}$ is not compact. Indeed,
$$
e^{-t}\left(1-\frac1n\right)\to e^{-t}\neq0,
\qquad n\to\infty.
$$
Thus, $Be^{tA}$ is not compact for any $t>0$. Finally, the perturbed operator is given by
$$
(A+B)x
=
\left(-\frac1n x_n\right)_{n\in\mathbb N}.
$$
Since $A+B$ is bounded, it generates the uniformly continuous semigroup
$$
e^{t(A+B)}x
=
\left(e^{-t/n}x_n\right)_{n\in\mathbb N},
\qquad t\geq0.
$$
For any fixed $n\in\mathbb N$, we have that $e^{-t/n}\to 0$ as $t\to\infty$. Moreover, notice that~$e^{-2t/n}|x_n|^2\leq |x_n|^2$ for all $n\in\mathbb N$ and $t\geq0.$
Therefore, by the dominated convergence theorem,
$$
\|e^{t(A+B)}x\|^2
=
\sum_{n=1}^{\infty}e^{-2t/n}|x_n|^2
\to0,
\qquad t\to\infty.
$$
Hence $\left(e^{t(A+B)}\right)_{t\geq0}$ is strongly stable. But,
$ \|e^{t(A+B)}\| =
\sup_{n\in\mathbb N}e^{-t/n} =1$ for  $t\geq0.$
Therefore the semigroup cannot satisfy an estimate of the form
$$
\|e^{t(A+B)}\|\leq Me^{-\omega t},
\qquad t\geq0,
$$
for some $M\geq1$ and $\omega>0$. Hence
$\left(e^{t(A+B)}\right)_{t\geq0}$ is strongly stable but not uniformly
exponentially stable.
\end{proof}

\appendix

\bibliography{biblio3}
\bibliographystyle{plain}

\end{document}